\documentclass{amsart}
\usepackage{graphicx}
\usepackage{caption}
\usepackage{amssymb}
\usepackage[matrix, arrow, curve]{xy}
\usepackage{enumerate}
\usepackage{hyperref}

\newtheorem{theorem}{Theorem}[section]
\newtheorem{lemma}[theorem]{Lemma}
\newtheorem{lemdef}[theorem]{Definition-Lemma}
\newtheorem{proposition}[theorem]{Proposition}
\newtheorem{cor}[theorem]{Corollary}

\theoremstyle{definition}
\newtheorem{definition}[theorem]{Definition}

\usepackage{tikz}
\usetikzlibrary{shapes}
\usetikzlibrary{fit}
\usetikzlibrary{decorations.pathmorphing}
\usetikzlibrary{calc,arrows}
\theoremstyle{remark}
\newtheorem{remark}[theorem]{Remark}

\numberwithin{equation}{section}

 \DeclareMathSymbol{\vartriangle}{\mathord}{AMSa}{"4D}
\DeclareMathSymbol{\triangledown}{\mathord}{AMSa}{"4F}

\begin{document}

\title{Laplacian matrices  and spanning trees of tree graphs}

\author{Philippe Biane}
\address{CNRS,  Institut Gaspard Monge
UMR  8049\\
Universit\'e Paris-Est
5 boulevard Descartes\\
77454 Champs-Sur-Marne\\
FRANCE }
\email{biane@univ-mlv.fr}
\author{Guillaume Chapuy}
\address{CNRS, IRIF UMR  8243\\
Universit\'e Paris-Diderot – Paris 7
Case 7014\\
75205 PARIS Cedex 13\\FRANCE }
\email{guillaume.chapuy@liafa.univ-paris-diderot.fr}
\thanks{Both authors acknowledge support from {\it Ville de Paris}, grant ``\'Emergences 2013, Combinatoire \`a Paris''. G.C. acnowledges support from {\it Agence Nationale de la Recherche}, grant ANR 12-JS02-001-01 ``Cartaplus''.}

\begin{abstract}
If $G$ is a strongly connected finite directed graph, the set $\mathcal{T}G$ of rooted directed spanning trees of $G$ is naturally equipped with a structure of directed graph: there is a directed edge from any spanning tree to any other obtained by adding an outgoing edge at its root vertex and deleting the outgoing edge of the endpoint.
Any Schr\"odinger operator on $G$, for example the Laplacian,  can be lifted canonically to $\mathcal{T}G$. We show that the determinant of such a lifted  Schr\"odinger operator admits a remarkable factorization into a product of    determinants of the restrictions of  Schr\"odinger operators on subgraphs of $G$
and we give a combinatorial description of the multiplicities using an exploration procedure of the graph.
A similar factorization can be obtained from earlier ideas of C. Athaniasadis,
 but this leads to a different expression of the multiplicities, as signed sums on which the nonnegativity is not appearent. 
We also provide a description of the block structure associated with this factorization.

 As a simple illustration we reprove a formula of Bernardi enumerating spanning forests of the hypercube, that is closely related to the graph of spanning trees of a bouquet. 
Several combinatorial questions are left open, such as giving a bijective interpretation of the results. 
\end{abstract}

\maketitle

\section{Introduction}
\label{sec:1}

Kirchoff's matrix-tree theorem relates the number of spanning trees of a graph to the minors of its Laplacian matrix. It has a number of applications in enumerative combinatorics, 
including Cayley's formula:
\begin{align}\label{eq:cayley}
|\mathcal{T}K_n| = n^{n-1},
\end{align}
counting rooted spanning trees of the complete graph $K_n$ with $n$ vertices 
and Stanley's formula:
\begin{align}\label{eq:hypercube}
|\mathcal{T}\{0,1\}^n| =\prod_{i=1}^n (2i)^{{n\choose i}},
\end{align}
for rooted spanning trees of the hypercube $\{0,1\}^n$, see \cite{Stanley}.
In probability theory, a variant of Kirchoff's theorem, known as the Markov chain tree theorem, expresses the invariant measure of a finite irreducible Markov chain in terms of spanning trees of its underlying graph (see \cite[Chap 4]{LP}, or \eqref{eq:mctt} below). 
An instructive proof of this result relies on lifting the Markov chain to a chain on the set of spanning trees of its underlying graph. In particular, this construction endows the set $\mathcal{T}G$ of spanning trees of any weighted directed graph $G$ with a structure of weighted directed graph. This construction is recalled in Section~\ref{sec:defs},
(the reader can already have a look at the example of Figure~\ref{fig:exampleGraph}).
In the recent paper \cite{B}, the first author 
conjectured that the number of spanning trees of 
 $\mathcal{T}G$ is given by a product of minors of the Laplacian matrix of the original graph $G$. 
In this paper, we prove this conjecture. 
More generally, given a Schr\"odinger operator on $G$, we will show (Theorem~\ref{thm:main}) that the determinant of a lifted Schr\"odinger operator on $\mathcal{T}G$ factorizes as a product of determinants of submatrices of the Schr\"odinger operator on $G$. In this factorization, only submatrices indexed by strongly connected subsets of vertices $W\subset V(G)$ appear, and the multiplicity $m(W)$ with which a given subset appears is described combinatorially via an algorithm of exploration of the graph $G$.

The case of the adjacency matrix (another special case of Schr\"odinger operator) was already studied by C. Athanasiadis who related the eigenvalues in the graph and in the tree graph (see~\cite{An}, or Section~\ref{sec:An}). As we shall see, this leads to a similar factorization of the characteristic polynomial as the one we obtain, and in fact the proof of~\cite{An} can easily be extended to any Schr\"odinger operator.
However the methods of~\cite{An}, whose proofs are based on a direct and elegant path-counting approach via inclusion-exclusion, lead to an expression of the multiplicities  as signed sums which are not apparently positive.
Our proof is of a different kind and
proceeds by constructing sufficiently many invariant subspaces of the Laplacian matrix of $\mathcal{T}G$. 
It is both algebraic and combinatorial in nature, but it leads to a positive description of the multiplicities. As a result our main theorem, or at least its main corollary, can be given a purely combinatorial formulation, which suggests the existence of a purely combinatorial proof. This is left as an open problem. Another combinatorial problem that we leave open 
concerns the definition of the multiplicities $m(W)$: in the way we define them, these numbers depend both on a total ordering of the vertex set of the graph, and on the choice of a ``base point'' in each subset $W$, but it follows from the algebraic part of the proof that they actually do not depend on these choices. This property is mysterious to us
and a direct combinatorial understanding of it would probably shed some light on the previous question.

Finally, we note that there exists a factorization for the Laplacian matrix of the line graph associated to a directed graph (see \cite{L}) that  looks similar to what we obtain here for the tree graph. The case of the tree graph is actually more involved.

\smallskip

The paper is structured as follows. In Section~\ref{sec:defs}, we state basic definitions and recall the construction of the tree graph.  We also present the results of Athanasiadis~\cite{An} and rephrase them from the viewpoint of the characteristic polynomial. Then in Section~\ref{sec:formula} we introduce the algorithm that defines the multiplicities $m(W)$, which enables us to state our main result for the Schr\"odinger operators (Theorem~\ref{thm:main}). We also state a corollary (Theorem~\ref{thm:mainSpanning}) that deals with spanning trees of the tree graph $\mathcal{T}G$, thus answering directly the question of~\cite{B}. In Section~\ref{proof}, we give the proof of the main result, that works, first, by constructing some invariant subspaces of the Schr\"odinger operator of $\mathcal{T}G$, then by checking that we have constructed sufficiently enough of them using a degree argument. Finally in Section~\ref{sec:examples} we illustrate our results by treating a few examples explicitly.

\smallskip

{\bf Acknowledgements.} When the first version of this paper was made public, we were not aware of the reference~\cite{An}. We thank Christos Athanasiadis for drawing our attention to it.
 G.C. also thanks Olivier Bernardi for an interesting discussion related to the reference~\cite{Bernardi}.

\begin{figure}[h]
\includegraphics[width=\linewidth]{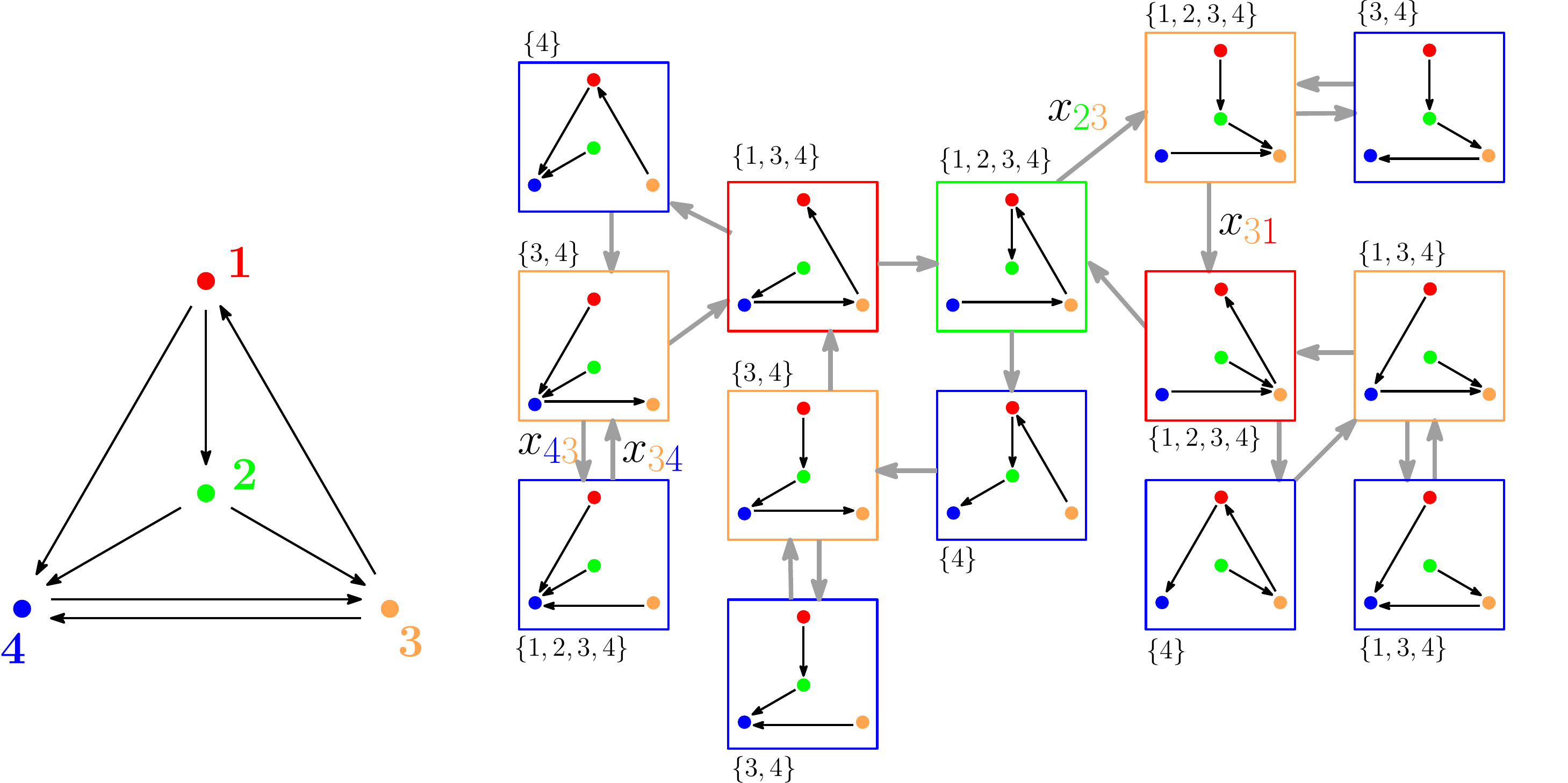}
\caption{A directed graph $G$ with $4$ vertices (Left), and the graph $\mathcal{T}G$ (Right). Each vertex of $\mathcal{T}G$ is one of the $14$ spanning trees of $G$. The weight of an edge in $\mathcal{T}G$ only depends on the root vertex of the two trees it links -- only certain edge weights are indicated on this picture. The subset of vertices indicated next to each spanning tree is the $\psi$-value returned by the algorithm of Section~\ref{sec:algo}.}\label{fig:exampleGraph}
\end{figure}

\section{Directed graphs and tree graphs}
\label{sec:defs}
In this section we set notations and recall a few basic facts.

\subsection{Directed graphs}
In this paper all directed graphs are finite and simple.
Let $G=(E,V)$ be a directed graph, with vertex set $V$ and edge set $E$. 
For each edge we denote $s(e)$ its \emph{source} and $t(e)$ its \emph{target}.
The graph $G$ is {\sl strongly connected} if for any pair of vertices $(v,w)$ there exists an oriented path from $v$ to $w$.

If $W\subset V$ then the graph $G$ induces a graph $G_W=(W,E_W)$ where $E_W$ is the set of edges $e$ with $s(e),t(e)\in W$.
A subset $W\subset V$ will be said to be \emph{strongly connected} if the graph $G_W$  is strongly connected.
A {\sl cycle} in $G$ is a path which starts and ends at the same vertex. The cycle is {\sl simple} if each vertex and each edge  in the cycle is traversed exactly once. 

\subsection{Laplacian matrix and Schr\"odinger operators}
For  a finite directed graph $G$, let $x_{e},e\in E$ be a set of indeterminates. 
The {\sl edge-weighted Laplacian} of the graph is the matrix $(Q_{vw})_{v,w\in V}$ given by
$Q_{vw}=x_e$ if $v\ne w$  $s(e)=v, t(e)=w$
(this quantity is $0$ if there is no such edge)
 and  $Q_{vv}=-\sum_{e:s(e)=v} x_e$.  

Let $y_{v},v\in V$ be another set of  variables and $Y$ be the diagonal matrix with  $Y_{vv}=y_v$. The associated {\sl Schr\"odinger operator with potential} $Y$ is the matrix $L=Q+Y$.
Observe that, if one specializes the variables $y_v$ to a common value $-z$, then $L=Q-zI$ and $\det(L)$ is the characteristic polynomial of $Q$ evaluated on $z$.

We will consider the space of functions  on $V$ with values in the field of 
rational fractions ${\mathbb F}_G={\bf C}(x_e; e\in E, y_v; v\in V)$, and the space of  measures on $V$ (again with with values in ${\mathbb F}_G$). These are vector spaces over the field ${\mathbb F}_G$.
The Schr\"odinger operator $L$ acts on  functions on the right by
$$L\phi(v)=\sum_{w}L_{vw}\phi(w),$$
and on measures on the left by 
$$\mu L(w)=\sum_{v}\mu(v)L_{vw}.$$

The space of measures has a basis given
by the $\delta_v,v\in V$ where $\delta_v$ is the measure putting mass 1 on $v$ and 0 elsewhere.

\subsection{A Markov chain} If the $x_e$ are positive real numbers, the matrix $Q$ is the generator of a continuous time Markov chain on $V$, with semigroup of probability transitions given by $e^{tQ}$. This chain is irreducible if and only if the graph $G$ is strongly connected. The function 1 is in the kernel of the action of $Q$ on functions, and this kernel is one-dimensional if and only if the chain is irreducible. Dually, if the chain is irreducible then there is a positive measure in the kernel of the action of $Q$ on measures (by the Perron-Frobenius theorem), which is unique up to a multiplicative constant. See for example~\cite{Seneta} for more on these classical results.

\subsection{Spanning trees} Let $G$ be a directed graph, 
an {\sl oriented spanning tree } of $G$ (or \emph{spanning tree} of $G$ for short) is a subgraph of $G$, containing all vertices, with no cycle, in which one vertex, called the {\sl root}, has outdegree 0 and the other vertices have outdegree 1.
If $\mathbf{a}$ is such a tree, with edge set $E_\mathbf{a}$, we denote 
\begin{equation}\label{arbres}\pi_\mathbf{a}=\prod_{e\in E_\mathbf{a}}x_e.\end{equation}
More generally, if $W\subset V$ is a nonempty subset, an {\sl oriented forest}  of $G$,
 rooted in $W$, is a subgraph of $G$, containing all vertices, with no 
 cycle and such that vertices in $W$ have  outdegree 0 while the other vertices have outdegree 1.
Again for a forest~$\mathbf{f}$, with edge set $E_\mathbf{f}$, we put 
\begin{equation}\label{forets}\pi_\mathbf{f}=\prod_{e\in E_\mathbf{f}}x_e.\end{equation}

The matrix-tree theorem states that, if $W\subset V$ and $Q^W$ is the matrix obtained from $Q$ by deleting rows and columns indexed by elements of $W$, then 
$$\det(Q^W)=\sum_{\mathbf{f}\in \mathcal F_W}\pi_\mathbf{f} $$
the sum being over oriented forests rooted in $W$.
In particular, in the Markov chain interpretation, an explicit formula for an invariant measure is given by
\begin{align}\label{eq:mctt}
\mu(v)=\sum_{\mathbf{a}\in T_v}\pi_\mathbf{a},
\end{align}
where the sum is over spanning oriented trees rooted at $v$. This statement is known, in the context of probability theory, as the Markov Chain Tree theorem, see~\cite[Chap. 4]{LP}.

It will be convenient in the following to use the notation $Q_W=Q^{V\setminus W}$ and  $L_W=L^{V\setminus W}$to denote the matrix extracted from the Laplacian or Schr\"odinger matrix  of $G$ by keeping only lines and columns indexed by elements of $W$.

\subsection{The tree graph $\mathcal{T}G$}\label{sec:wlap}
Let $G=(E,V)$ be a finite directed graph and $\mathbf{a}$
an oriented spanning tree  of $G$ with root $r$. For  an edge $e\in V$ with $s(e)=r$,  let $\mathbf{b}$ be the subgraph of $G$  obtained by adding edge $e$ to $\mathbf{a}$ then deleting the edge coming out of $t(e)$ in $\mathbf{a}$. See Figure~\ref{fig:treeGraphRules}. It is easy to check that $\mathbf{b}$ is an oriented spanning tree  of $G$, with root $t(e)$.
\begin{figure}[h]
\includegraphics[width=0.6\linewidth]{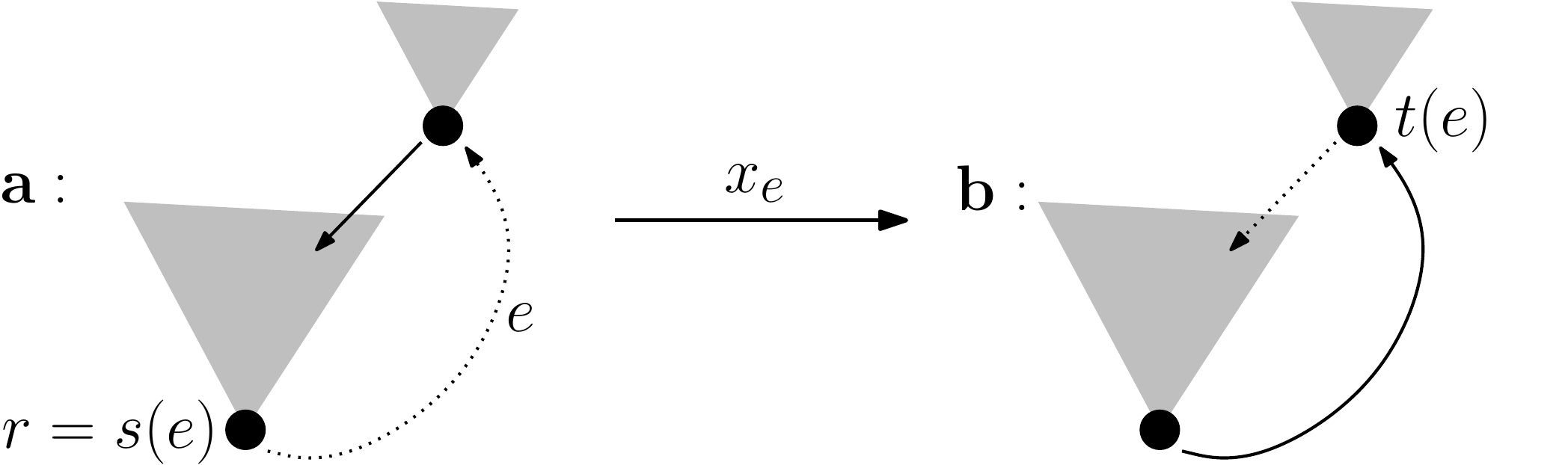}
\caption{An edge $\mathbf{a}\rightarrow \mathbf{b}$ in the tree graph $\mathcal{T}G$. It is associated with the edge weight $x_e$.}\label{fig:treeGraphRules}
\end{figure}

The {\sl  tree graph} of $G$, denoted $\mathcal T G$, is the directed graph whose vertices are the oriented spanning trees of $G$ and whose edges are 
obtained by the  previous construction, i.e. for each pair $\mathbf{a},\mathbf{b}$ as above  we obtain an edge of $\mathcal T G$  with source $\mathbf{a}$ and target $\mathbf{b}$. We will denote 
$\mathcal{T}V$ the set of vertices of $\mathcal{T}G$, in other words, $\mathcal{T}V$ is the set of oriented spanning trees of $G$.
Figure~\ref{fig:exampleGraph} gives a full example of the construction.
One can prove that the graph $\mathcal T G$ is strongly connected if $G$ is, see for example \cite{AT}. Moreover the graph $\mathcal T G$ is simple and has no loop.
There is a natural map $p$ from $\mathcal T G$ to $G$ which maps each  vertex of $\mathcal T G$,
 which is  an oriented spanning tree of $G$, to its root, and  maps each edge  of $\mathcal T G$   to the edge $e$ of $G$ used for its construction.

We assign weights to the edges and vertices of $\mathcal{T}G$ as follows: we give the weight $x_{e}$ to any edge $e'$ of $\mathcal T G$ 
such that $p(e')=e$ and we give the weight $y_v$ to the tree $\bf a$ if its root is $v$.

This leads to a weighted Laplacian and a Schr\"odinger operator for $\mathcal T G$, which we denote respectively by $\mathcal Q$ and $\mathcal L$. More precisely,  
 $\mathcal Q$ is the matrix with rows and columns indexed by the oriented spanning trees of $G$ such that 
\begin{align*}
\mathcal Q_{\mathbf{a}\mathbf{c}}=&0\text{ if $\mathbf{a}\ne \mathbf{c}$ and $\mathbf{a}\mathbf{c}$ is not an edge of $\mathcal T G$ }
\\
  \mathcal Q_{\mathbf{a}\mathbf{b}}=&x_{e}\text{ if $\mathbf{a}\mathbf{b}$ is an edge of $\mathcal T G$ and $e$ is the edge of $\mathbf{b}$ going out the root of $\mathbf{a}$.}
\\
\mathcal Q_{\mathbf{a}\mathbf{a}}=&-\sum_{\mathbf{b}\ne \mathbf{a}}\mathcal Q_{\mathbf{a}\mathbf{b}}.
\end{align*}

Similarly, $\mathcal Y$ is the diagonal matrix indexed by ${\mathcal T}V$ with 
${\mathcal Y}_{\bf aa}=y_{root({\bf a})}$ and 
$$\mathcal L=\mathcal Q+\mathcal Y.$$
See~\cite{AT} or \cite{LP} for more on the matrix $\mathcal Q$ in a context of probability theory. 
In \cite{B} the first author proved that there exists a polynomial $\Phi_G$ in the variables $x_e$ such that, for any oriented spanning tree $\mathbf{a}$ of $G$, one has
\begin{eqnarray}\label{eq:polbiane}
\det(\mathcal Q^{\mathbf{a}})=\pi_\mathbf{a}\Phi_G.
\end{eqnarray}
In the same reference it was conjectured that $\Phi_G$ is a product of symmetric minors of the matrix $Q$ (i.e. a product of polynomials of the form $\det(Q_W)$). In this paper we prove this conjecture and provide an explicit formula for $\Phi_G$ (Theorem~\ref{thm:mainSpanning}). Actually we deduce this from a more general result which computes the determinant of $\mathcal L$ as a product of determinants  of the matrices $L_W$ (Theorem~\ref{thm:main}).
These results will be stated in Section~\ref{sec:formula} and proved in Section~\ref{proof}.
The example of the tree graph of a cycle graph was investigated in \cite{B} and we will explain in Section~\ref{sec:examples} how it follows from our general result. 

\subsection{Structure of the tree graph}

Before we state and prove the main theorem of this paper, we give here some elementary properties of the tree graph, which might be of independent interest. 
These properties will not be used in the rest of the paper.

We start with the following simple observation:
for any directed path $\pi$ in the graph $G$, starting at some vertex $v$, and any oriented spanning tree $\mathbf{a}$ rooted at $v$, there exists a unique path starting at $\mathbf{a}$ in $\mathcal T G$ which projects onto $\pi$. Thus the graph $\mathcal T G$ is a 
covering graph of $G$.

If $\mathbf{a}\to \mathbf{b}$ is an edge of $\mathcal T G$, then the union of the edges of $\mathbf{a}$ and $\mathbf{b}$ is a graph with a simple cycle $C$, containing the roots of $\mathbf{a}$ and $\mathbf{b}$, and a forest, with edges disjoint from the edges of $C$, rooted on the vertices of $C$.
The cycle $C$ is the union of the path from the root of $\mathbf{b}$ to the root of $\mathbf{a}$ in the tree $\mathbf{a}$ with the edge from the root of $\mathbf{a}$ to the root of $\mathbf{b}$ in $\mathbf{b}$. If we lift the cycle $C$ in $G$ to a path $\mathcal T C$ in $\mathcal T G$, starting from $\mathbf{a}$, we  get a cycle in $\mathcal T G$, which projects 
bijectively onto the simple cycle $C$.
The cycle $C$, and thus $\mathcal T C$ is completely determined by the edge $\mathbf{a}\mathbf{b}$ in $\mathcal T G$, moreover for any edge in $\mathcal T C$, the associated cycle is again 
 $\mathcal T C$. Conversely, if $C$ is a simple cycle of $G$, and $\mathbf{f}$ a forest rooted at the vertices of $C$, then the trees obtained from $C\cup \mathbf{f}$ by deleting an edge of $C$ form a simple cycle in $\mathcal T G$ which lies above $C$. We deduce:
\begin{proposition}
The set of edges of $\mathcal T G$ can partitionned into edge-disjoint simple cycles, which project onto simple cycles of $G$.
If $C$ is a simple cycle of $G$, with vertex set $W$, then the number of simple cycles of $\mathcal T G$ lying above $C$ is equal to the number of forests rooted in $W$.
\end{proposition}
In particular, to any outgoing edge  of $\mathbf{a}$ in $\mathcal{T}G$ one can associate the incoming edge of the cycle to which it belongs, and this gives a bijection between incoming and outgoing edges of $\mathbf{a}$.  
An immediate corollary is
\begin{cor}
The graph $\mathcal T G$ is Eulerian: The  number of outgoing or incoming edges of a vertex $\mathbf{a}$ are both equal  to the number of outgoing edges of the root of $\mathbf{a}$ in $G$.
\end{cor}
The previous discussion also implies that the measure on vertices of $\mathcal{T}G$ that gives mass $\pi_\mathbf{a}$ to each tree $\mathbf{a}$ is an invariant measure on $\mathcal{T}G$, \textit{i.e.} one has $\pi R =0$.  By using the projection map $p$, it follows that the measure $\mu$ on $V$ given by~\eqref{eq:mctt} satisfies $\mu Q=0$, which gives a simple proof of the Markov Chain tree theorem, see \cite{AT}.

\section{A formula for the determinant of the Schr\"odinger operator}
\label{sec:formula}

We use the same notation as in the previous sections, in particular $V$ is the vertex set of the directed graph $G$, the weighted Laplacian of $G$ is $Q$, its Schr\"odinger operator is $L$, the graph of spanning trees is denoted $\mathcal T G$ and the weighted Laplacian and Schr\"odinger operators of $\mathcal T G$, as in Section~\ref{sec:wlap}, are denoted by $\mathcal Q$ and $\mathcal L$. We assume that $G$ is strongly connected.

\subsection{Eigenvalues of the adjacency matrix, according to Athanasiadis~\cite{An}}\label{sec:An}
If the weights $y_v$ are set to $y_v=-Q_{vv}=\sum_{e:s(e)=v} x_e$, then the Schr\"odinger operator  becomes the adjacency matrix of the graph $G$. We denote it by $M$. It is easy to see that in this case the lifted Schr\"odinger operator $\mathcal{M}$ is the adjacency matrix of the graph $\mathcal{T}G$.
In \cite{An} C.~Athanasiadis proves the following result about eigenvalues of the matrix $\mathcal M$. 

\begin{proposition}[\cite{An}]
The eigenvalues of the adjacency matrix $\mathcal M$ are eigenvalues of the matrices $M_X; X\subset V$. For such an eigenvalue $\gamma$, if $m_X(\gamma)$ denotes its multiplicity in $M_X$, then
its multiplicity in $\mathcal M$ is
$$\sum_{X\subset V}m_X(\gamma)\det(\Gamma_X-I)$$ where $\Gamma_X$ is the matrix $M_X$ with all variables $x_e$ equal to -1.
\end{proposition}
The previous theorem implies the following equation
$$\det(zI-\mathcal M)=\prod_{X\subset V}\det(zI-M_X)^{l(X)}$$
where $l(X)=\det(\Gamma_X-I)$. Observe however that the multiplicities $l(X)$ can be negative in this equation. In order to get nonnegative multiplicities, we will use the  following fact which is easy to check:
 for any
$X\subset V$, if we let $X=\cup_iW_i$ be its decomposition into strongly connected components, then
the graph induces an order relation between the $W_i$ from which one deduces the factorization
$$\det(zI-M_X)=\prod_i\det(zI-M_{W_i}).$$
It follows that 
\begin{equation}\label{n(X)}
\det(zI-\mathcal M)=\prod_{W}\det(zI-M_W)^{n(W)}
\end{equation}
where the product is over strongly connected subsets $W\subset V$
and 
\begin{equation}\label{supset}
n(W)=\sum_{X\supset_{sc} W}l(X)
\end{equation}
where $X\supset_{sc} W$ means that $W$ is a strongly connected component of $X$. As we will see later
(Lemma~\ref{W_irred}),
the polynomials $\det(zI-M_W)^{n(W)}$ for $W$ strongly connected are distinct prime polynomials, therefore the formula \ref{n(X)} uniquely defines the multiplicities $n(X)$ which therefore are nonnegative integers. This property however is not apparent from the formula \ref{supset}.

In this paper, we will generalize this result to the case of Schr\"odinger operators and give another expression for the multiplicities, as the cardinality of a set of combinatorial objects (hence the nonnegativity will be apparent). We will also explicitly exhibit a block decomposition of the matrix $\mathcal L$ that underlies the factorization of the characteristic polynomial.

Although Athanasiadis's results were stated for adjacency matrices, his proof actually extends easily to the more general case of Schr\"odinger operators which we consider here (with the same multiplicities). However the link between the two approaches is yet to be understood.

\subsection{The exploration algorithm}\label{sec:algo} 
Our formula for the determinant of $\mathcal L$ (given in Theorem~\ref{thm:main}) involves certain combinatorial quantities defined through an algorithmic exploration of the graph. The exploration algorithm associates to any spanning tree $\mathbf{a}$ of $G$ two subsets of vertices of $G$, denoted by $\phi(\mathbf{a})$ and $\psi(\mathbf{a})$. 
Roughly speaking, the algorithm performs a breadth first search on the graph $G$, but only the vertices that are discovered along edges belonging to the tree $\mathbf{a}$ are considered as explored. Vertices discovered along edges not in $\mathbf{a}$ are immediately ``erased''. This may prevent the algorithm from exploring the whole vertex set and, at the end, we call $\phi(\mathbf{a})$ the set of explored vertices.
The set $\psi(\mathbf{a})$ is the  strongly connected component of the root vertex in $\phi(\mathbf{a})$.

 We now describe more precisely the algorithm. Because it is based on breadth first search, our algorithm depends on an ordering of the vertices of $V$. This ordering can be arbitrary but it is important to fix it once and for all:
\begin{quote}\it From now on  we fix a total ordering of the vertex set $V$ of $G$.
\end{quote}
In particular on examples and special cases considered in the paper, if the vertex set is an integer interval, we will equip it with the natural ordering on integers without further notice (this is the case for example on Figure~\ref{fig:exampleGraph}).

\noindent\rule{\textwidth}{0.7pt}
\begin{tabbing} 
{\bf Exploration algorithm.}\\
{\bf Input:}\ \ \ \ \   \=A spanning tree $\mathbf{a}$ of the directed graph $G=(V,E)$, rooted at $v$. \\
{\bf Output:} 
\>A subset of vertices $\phi(\mathbf{a})\subset V$;\\
\>A subset of vertices $\psi(\mathbf{a})\subset \phi(\mathbf{a})$, such that $G_{|\psi(\mathbf{a})}$ is strongly connected.
\end{tabbing}
\begin{tabbing} 
\noindent {\bf Running variables: }\= - a set $A$ of vertices of $G$;\\
\> - an ordered list ${\bf L}$ of edges of $G$ (first in, first out);\\
\> - a set $F$ of edges of $G$.
\end{tabbing}
\noindent {\bf Initialization:}
Set $A:=\{v\}$, $F:=\{e\in E|s(e)\ne v\}$, and let ${\bf L}$ be the list of edges of $G$ with target $v$, ordered by increasing source.\\ 
\noindent {\bf Iteration:} 
While ${\bf L}$ is not empty, pick the first edge $e$ in ${\bf L}$ and let $w$ be its source:
\begin{tabbing} 
\ \ \ \ \= \noindent {\bf If} \= $e$ belongs to the tree $\mathbf{a}$: \\
\>\> add 
$w$ to $A$;\\
\>\>delete all edges with source $w$ from ${\bf L}$;\\
\>\> append at the end of ${\bf L}$ all the edges in $F$ with target $w$, by increasing source.\\
\> {\bf else} \\ 
\>\> delete from ${\bf L}$ and $F$ all the edges with source or target $w$ in $E$.\\ 
\>\> (in this case we say that the vertex $w$ has been \emph{erased}) 
\end{tabbing}

\noindent {\bf Termination:} 
We let $\phi(\mathbf{a}):=A$ be the terminal value of the evolving set $A$. The directed graph $G$ induces a directed graph on $\phi(\mathbf{a})$, and we let $\psi(\mathbf{a})$ be the strongly connected component of $v$ in this graph.\\
\noindent\rule{\textwidth}{0.7pt}

 Observe that if a vertex $w$ is picked up by the algorithm at some iteration, it will not appear again, this implies that the algorithm always stops after a finite number of steps.
We refer the reader to Figure~\ref{fig:exampleAlgo} for an example of application of the algorithm. The reader can also look back at Figure~\ref{fig:exampleGraph} on which, for each spanning tree $\mathbf{a}$, the value of the set $\psi(\mathbf{a})$ is indicated.
\begin{figure}[h]
\includegraphics[width=0.7\linewidth]{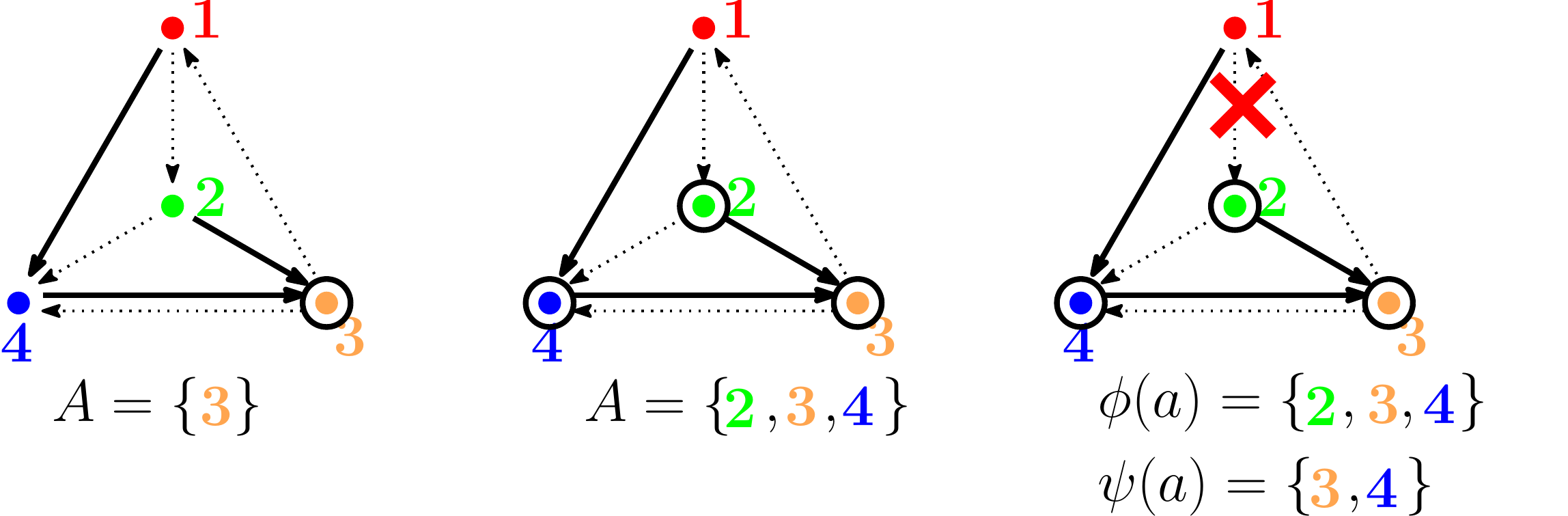}
\caption{\textit{Left:} in plain edges, a spanning tree $\mathbf{a}$ of the graph of Figure~\ref{fig:exampleGraph}. We initialize the set $A$ to $\{3\}$ since $3$ is the root of $\mathbf{a}$. \textit{Center:} at the first two iterations of the main loop of the algorithm, we consider the edges $(2,3)$ and $(4,3)$, that belong to the tree $\mathbf{a}$: the vertices $2$ and $4$ are thus added to the set $A$. \textit{Right:} at the next iteration, we consider the edge $(1,2)$ that does not belong to $\mathbf{a}$. The vertex $1$ is thus erased. The set $A$ will not evolve until the termination step, and we thus get $\phi(\mathbf{a})=\{2,3,4\}$. The strongly connected component of $3$ inside $\{2,3,4\}$ in the original graph is $\{3,4\}$, which gives the value of $\psi(\mathbf{a})$.}
\label{fig:exampleAlgo}
\end{figure}

With the exploration algorithm, we can now define the multiplicities that are necessary to state our  main theorem.
\begin{definition} Let $W$ be a strongly connected subset of $V$, and $w\in W$. The  \emph{multiplicity of $W$ at $w$}  is the number 
 $m(W,w)$ of oriented spanning trees $\mathbf{a}$ rooted at $w$ such that $\psi(\mathbf{a})=W$.
\end{definition}
\noindent For any $v \in V$, there exists a unique tree $\mathbf{a}_{V,v}$ rooted at $v$ such that $\psi(a_{V,v})=V$. This tree is obtained by performing a {\sl breadth first search} on $G$ starting from $v$ and keeping the edges of first discovery of each vertex. We thus have:
\begin{lemma}\label{bfs} For any $v\in V$ one has
$m(V,v)=1$.
\end{lemma}

More generally, we will prove in Section~\ref{sectheta} the following fact
\begin{lemdef}\label{mW}
For any strongly connected subset $W \subset V$, the multiplicity  $m(W,w)$  depends neither on $w\in W$ nor on the ordering of the elements of $V$.
We will call $m(W)$ this common value.
\end{lemdef}
\begin{proof} See Section~\ref{sectheta}.
\end{proof}

\subsection{Main result}
 Our main result is the following theorem.
\begin{theorem}\label{thm:main}
Let $G$ be a strongly connected directed graph. Then
the determinant of the lifted Schr\"odinger operator on $\mathcal{T}G$ is given by:

\begin{equation}\label{main1}\det(\mathcal L)=\prod_{\stackrel{
W\subseteq V}{W s.c.}}\det(L_W)^{m(W)}\end{equation}
where the product is over all  strongly connected subsets $W\subseteq V$.
\end{theorem}

From the previous result we will deduce the following formula for $\Phi_G$. Recall that we defined $\pi_{\bf a}$ in (\ref{arbres}) as the product over the weights of the edges of a tree $\bf a$ and similarly $\pi_{\bf f}$ (\ref{forets}) for a forest. Analogously one  defines the weight of a spanning tree of 
$\mathcal{T}G$ as the product of the weights of its edges. We define the polynomials $ F_{G}$, $F_{\mathcal{T}G}$ and 
$\Psi_W$ as the sums of these weights over, respectively, spanning trees of $G$, of $\mathcal{T}G$, and of forests of $G$ rooted in $W$. The Markov chain tree theorem implies that the generating function of the spanning trees of a graph is the coefficient of the term of degree 1 in the characteristic polynomial of the Laplacian matrix. Using this fact and Theorem \ref{thm:main} we obtain the following result.
\begin{theorem}[Spanning trees of the tree graph]\label{thm:mainSpanning}
The generating polynomial $F_{\mathcal{T}G}$ of spanning trees of the tree graph is given by
\begin{eqnarray}\label{eq:ratio}
F_{\mathcal{T}G} = \Phi_G F_{G},
\end{eqnarray}
 where
\begin{equation}\label{main2}
\Phi_G=\prod_{\stackrel{W\subsetneq V}{W s.c.}}(\Psi_{V\setminus W})^{m(W)},
\end{equation}
where the product is over all proper strongly connected subsets $W\subsetneq V$.
\end{theorem}

Note that from \eqref{eq:polbiane}
and the matrix-tree theorem, Theorem~\ref{thm:mainSpanning} 
also gives a formula for spanning trees of $\mathcal{T}G$ rooted at a particular spanning tree $\mathbf{a}$.

Note also that summing over all trees $\bf a$ in \eqref{eq:polbiane} and using the matrix-tree theorem, we see that the constant $\Phi_G$ in \eqref{eq:polbiane} is indeed the same as the one in~\eqref{eq:ratio}.

\begin{remark}
Both sides of Equation (\ref{main2}) have a natural combinatorial meaning; the left hand side is a generating function for spanning trees of $\mathcal T G$, while the right hand side is the generating function of some tuples of forests on $G$. It would be interesting to have a direct combinatorial proof of this identity.
\end{remark}

As an example, on the graph $G$ of Figure~\ref{fig:exampleGraph}, there are $7$ strongly connected proper subsets of vertices and we have: $m(\{1\})=m(\{2\})=m(\{3\})=0$, $m(\{4\})=3$, $m(\{3,4\})=2$, $m(\{1,2,3\})=0$, $m(\{1,3,4\})=1$, and $m(V)=1$.
It follows that the characteristic polynomial of the Schr\"odinger operator  of the graph $\mathcal{T}G$ in this case is given by
$$\det(\mathcal{L})=\det(L_4)^3\det(L_{3,4})^2\det(L_{1,3,4})
\det(L).
$$
This identity can of course also be checked by a direct computation.

\section{Proof of the main results}\label{proof}

In this section we prove the main results. We assume as above that $G$ is strongly connected and we use the same notation as in previous sections.

\subsection{Polynomials}
In order to prove Theorem~\ref{thm:main} we will show that each factor in \eqref{main1} appears with, at least, the wanted multiplicity and conclude by a degree argument. We start by showing that these factors are irreducible. 
\begin{lemma}\label{W_irred}
If $W\subset V$ is a proper  strongly connected subset then the  polynomial $\det(L_W)$ is irreducible as a polynomial in the variables $(x_e)_{e\in E}; (y_v)_{v\in W}$.
\end{lemma}

\begin{proof}
First we note that $\det(L_W)$ is a homogeneous polynomial, and it has degree at most one in each of the variables $x_e, e\in E, s(e)\in W, (y_v)_{v\in W}$. Moreover, by Kirchhoff's theorem, its term of 
total
degree 0 in the $y$ variables is the generating function of forests rooted in $V\setminus W$, which is nonzero since $W$ is strongly connected and proper. In particular, the polynomial is not divisible by any of the $y_v$.
By expanding the determinant $\det(L_W)$ along the row indexed by some $w\in W$, we see that  for each $w$, in each  monomial  of $\det(L_W)$ there is at most one factor $x_e$ with $s(e)=w$. It follows that, for each $w$, as a polynomial in the variables $(x_e;s(e)=w)$, 
the polynomial $\det(L_W)$ has degree 1.

Now assume that $\det(L_W)=AB$ is a nontrivial factorization into homogeneous polynomials  then, from the previous point, for each $w$ the polynomial $AB$ is a factorization of a degree one polynomial in $(x_e;s(e)=w)$. It follows that there must exist a partition of $W=X\uplus {X'}$ where $A$ is a polynomial in the $y_v$ and in the variables $x_e$ with $s(e)\in X$, while $B$ is a polynomial in the $y_v$ and in the variables $x_e$ with $s(e)\in {X'}$;
note that this partition is non trivial since $\det(L_W)$ is not divisible by any $y_v$.
Moreover every monomial of $\det(L_W)$ can be written in a unique way as a product of a monomial appearing in $A$ and a monomial appearing in $B$.
Putting all variables $x_e$ with $s(e)\in X$ to zero we see that
$$
\left.\det(L_W) \right|_{x_e=0, s(e)\in X}
=\det(L_{X'})\prod_{v\in X}y_v = A(y,0)B.
$$
The same can be done for ${X'}$ and we obtain that 
$
\det(L_X) 
\det(L_{X'})
=h(y)\det(L_W)
$
 where $h(y)=A(y,0)B(y,0)\prod_{v\in W}y_v^{-1}$ is a Laurent polynomial. By looking at the top coefficient in the $y_v$ on both sides it follows that $h =1$, hence  
$$
\det(L_W)=
\det(L_X) 
\det(L_{X'})
.$$
Since the graph $G_W$ is strongly connected there exists
a spanning tree $\mathbf{a}$ of $W$ rooted in some vertex $x\in X$; in the corresponding monomial term of $\det(L_W)$ there is a factor $x_e$ with $s(e)=x'$ for each $x'\in {X'}$, since each vertex of ${X'}$ has an outgoing edge in the tree~$\mathbf{a}$.  The corresponding monomial therefore appears  in $\det(L_{X'})$, and we note that each variable $x_e$ appearing in this monomial is such that $t(e)\in W$. The argument can be repeated for $X$ and we deduce that there exists a monomial term in  $\det(L_W)=\det(L_X)\det(L_{X'})$ which is a product of variables $x_e$ which are all such that $t(e)\in W$; this monomial does not correspond to a forest by a simple counting argument, hence a contradiction.
\end{proof}

\subsection{The case of the full minor.}  The space of functions on $\mathcal T G$ which depend only on the root
of the tree (i.e. functions $F$ such that $F(\mathbf{a})=F(\mathbf{b})$ if $p(\mathbf{a})=p(\mathbf{b})$) is  invariant by the action of $\mathcal L$ on functions, moreover the restriction of $\mathcal L$ to this subset is clearly equivalent to the action of $L$ on the functions on $V$ by the obvious map. 
 Dually the matrix $\mathcal L$ leaves invariant the space of measures $\mu$ such that $\mu(p^{-1}(v))=0$ for all $v\in V$. The action of $\mathcal L$ on the quotient of $meas(\mathcal{T}G)$ by this subspace is isomorphic to the action of $L$ on $meas(G)$. From either of these remarks, we deduce

\begin{lemma}\label{div1}
The polynomial $\det(L)$ divides $\det(\mathcal L)$.
\end{lemma}

\subsection{Boundary and erased vertices}\label{sec:boundary}
 We make some remarks on the algorithm of Section \ref{sec:algo}. Once we have applied the algorithm to a given tree $\mathbf{b}$, with output $W=\psi(\mathbf{b})$,  we can distinguish
 several subsets of vertices: 
\begin{enumerate}
\item the set $Z=\phi(\mathbf{b})$, which is the set of vertices of a subtree of $\mathbf{b}$;
\item the set $W=\psi(\mathbf{b})$, which is the set of vertices of a subtree of the previous one;
\item the set $Y=V\setminus Z$;
\item  the set of {\sl erased points} which are the vertices which have been erased when applying the algorithm. 
\item the set of {\sl boundary points,
} which are the vertices in $Y$ having
an outgoing edge with target in $Z$.
\end{enumerate} 

\begin{lemma}\label{erase}
The sets of boundary points and of erased points coincide.
\end{lemma}
\begin{proof}
In an iteration of the algorithm, any vertex which has been added to the set $A$ has all its outgoing edges suppressed, therefore it cannot be erased in a subsequent iteration. It follows that, 
if  a vertex has been erased during the algorithm, then it does not belong to $Z$ and it is the source of some edge with
target in $Z$ therefore it is a boundary point. 
Conversely if $v$ is a boundary point let $z\in Z$ be the first vertex, among the targets of an outgoing edge of $v$, which is scanned by the algorithm, then the edge from $v$ to $z$ does not belong to the tree $\mathbf{b}$ (if it did, $v$ would be in $Z$), therefore $v$ is erased when
one applies the algorithm at $z$. 
\end{proof}

\subsection{Constructing the invariant subspaces}\label{sec:subspace}
Let $W\subset V$ be a strongly connected proper subset. In this section and the next we will construct $m(W)$ complementary
vector spaces that are invariant by $\mathcal L$ and on which $\mathcal L$ acts as the matrix $L_W$. This will be the main step towards proving~\eqref{main1}. This construction goes in two steps: we first build a space of measures that is not invariant (this section, \ref{sec:subspace}) and we then construct a quotient of this space by imposing suitable ``boundary conditions'' that make the quotient space invariant (Section~\ref{sectheta}).

For every pair $(\mathbf{a},\mathbf{f})$ formed of a spanning oriented tree $\mathbf{a}$ of $W$ and an oriented forest $\mathbf{f}$ rooted in $W$,  let us call $\mathbf{a}\times \mathbf{f}$ the oriented spanning tree of $V$, rooted in the root of $\mathbf{a}$, obtained by taking the union of the edges of $\mathbf{a}$ and $\mathbf{f}$. Let us  denote by $\mathcal T W$ the set of oriented spanning trees of $W$ and $\mathcal F W$ the set of oriented forests rooted in $W$. We thus have an injection $\mathcal  TW\times \mathcal  FW\to \mathcal T V$ and correspondingly a linear map from 
$meas(\mathcal  TW)\otimes meas(\mathcal FW )\to meas (\mathcal T V)$.
Fix some forest $\mathbf{f}$ as above and consider the matrix $\mathcal L_{(W)}$ obtained from $\mathcal L$ by keeping only the rows and columns corresponding to oriented spanning trees of $V$ of the form
$\mathbf{a}\times \mathbf{f}$ where $\mathbf{a}$ is some spanning oriented tree of $W$. It is easy to see that this matrix, considered as a matrix indexed by elements of $\mathcal T W$, does not depend on the forest $\mathbf{f}$, but only on $W$. It differs from the matrix $\mathcal L$ constructed from the graph $G_W$ by some diagonal terms corresponding to the fact that there exists edges in $E$ with source in $W$ and target in $V\setminus W$.
The matrix $\mathcal L_{(W)}$ acts on functions on 
$\mathcal TW$ 
and on measures on 
$\mathcal TW$,
and it is easy to see that for its action on measures, the space of measures on 
$\mathcal TW$
 such that $\mu(p^{-1}(w))=0$ for every vertex $w\in W$ is an invariant subspace of measures. The action of  $\mathcal L_{(W)}$ on the quotient of $meas(\mathcal{T}W)$ by this subspace is isomorphic to the action of  $L_W$ on $meas(W)$.

\subsection{Boundary conditions and proof of Theorem~\ref{thm:main}}\label{sectheta}
 The subspace of measures 
$meas(\mathcal  TW)\otimes meas(\mathcal FW )\subset meas (\mathcal T V)$ is not invariant by the action of  $\mathcal L$ on measures but we will see that by modifying it and imposing suitable "boundary conditions"  we will obtain 
an invariant subspace.
For this let us consider a vertex $w\in W$ and a tree  $\mathbf{b}$, rooted at $w$,  such that  $\psi(\mathbf{b})=W$. The tree $\mathbf{b}$ is of the form $\mathbf{a}\times \mathbf{f}$ considered above, moreover the tree $\mathbf{a}$ depends only on $W$ and $w$, since it coincides with the breadth-first search exploration tree on $W$ (similarly as in Lemma \ref{bfs}). To emphasize this fact we use the notation $\mathbf{a}=\mathbf{a}_{W,w}$. The set of trees $\mathbf{b}$ rooted at $w$ and such that $\psi(\mathbf{b})=W$ is equal to $\mathbf{a}_{W,w}\times F_{W,w}$ where $F_{W,w}$ is some set of forests rooted in $W$, with $|F_{W,w}|=m(W,w)$. As indicated by the notation, the set $F_{W,w}$ may depend on both $W$ and $w$.

Let us fix $\mathbf{f}\in F_{W,w}$ and consider the set $\frak E_{\mathbf{b}}$ of vertices erased when running the algorithm on the tree~$\mathbf{b}=\mathbf{a}_{W,v}\times\mathbf{f}$.
A vertex $v$ is  erased when it is the source of some edge $e(v)$ considered in the algorithm, which is not in $\mathbf{b}$ and which is scanned before the edge of $\mathbf{b}$ going out of~$v$.
For a subset $\frak F\subset \frak E_{\mathbf{b}}$ let $\mathbf{f}_{\frak F}$ be the graph obtained by replacing in $\mathbf{f}$, for each erased vertex $v\in \frak F$, the edge going out of $v$  by the edge $e(v)$.
\begin{lemma} For each $\frak F\subset \frak E_{\mathbf{b}}$ the graph $\mathbf{f}_{\frak F}$ is a forest rooted in $W$
\end{lemma}
\begin{proof}
It suffices to observe that each vertex of $V\setminus W$ has outdegree $1$ and that, by construction, from any such vertex there is directed path going to $W$.
\end{proof}

For $\mathbf{f}\in F_{W,w}$, let 
$\nu_{\mathbf{f}}$ be the measure on $\mathcal{F}W$ defined by
\begin{align}\label{eq:defnu}
\nu_{\mathbf{f}}=\sum_{\frak F\subset \frak E_{\mathbf{b}}}(-1)^{|\frak F|}\delta_{f_{\frak F}},
\end{align}
with $\mathbf{b}=\mathbf{a}_{W,v}\times\mathbf{f}$.

\begin{lemma}The measures $  \nu_{\mathbf{f}}$ for $\mathbf{f}$ in $F_{W,v}$  are linearly independent.
\end{lemma}
\begin{proof}
First, construct a gradation on the set of spanning trees of $G$ rooted at $w$ as follows. If $\mathbf{a}$ is such a tree, let $v_0=w,v_1, \ldots v_k$ be the list of elements of $\phi(\mathbf{a})$ (the set of non-erased vertices) in the order they are discovered by the algorithm running on $\mathbf{a}$. We let $l_0,\ldots,l_k$ be the number of incoming edges of $v_0,\ldots,v_k$ from the set $\phi(\mathbf{a})$. 
This construction associates to any tree $\mathbf{a}$ rooted at $w$ a finite sequence $l_0,\ldots,l_k$ of integers. We equip the set of all sequences with the lexicographic order, which induces a gradation on the set of trees rooted at $w$.

Now, if $\mathbf{f}\in F_{W,w}$ and $\frak{F} \neq \emptyset$, then the tree $\mathbf{a}_{w,W}\times \mathbf{f}_{\frak{F}}$ is strictly higher in the gradation than $\mathbf{a}_{w,W}\times \mathbf{f}$. Indeed the origin of the first edge $e(v)$ for $v\in \mathfrak{F}$ that is considered by the algorithm belongs to the set $\phi(a_{w,W}\times \mathbf{f}_{\frak{F}})$ but not to $\phi(a_{w,W}\times \mathbf{f})$, which shows that at the first index where the degree sequences differ, the one corresponding to $\mathbf{a}_{w,W}\times \mathbf{f}_{\frak{F}}$ takes a larger value -- hence it is larger for the lexicographic order.

This shows that the transformation \eqref{eq:defnu} expressing the measures $\{\nu_{\mathbf{g}}, \mathbf{g} \in F_{W,w}\}$ in the basis $\{\delta_\mathbf{f}, \mathbf{f} \in \mathcal{F}W\}$ is given by a matrix of full rank: indeed, provided we order rows and columns by any total ordering of $\mathcal{F}W$ that extends the gradation defined by $\mathbf{f}<\mathbf{g}$ if $(\mathbf{a}_{W,w}\times \mathbf{f}) < (\mathbf{a}_{W,w}\times \mathbf{g})$, we obtain a strict upper staircase matrix. 
\end{proof}

It follows from the last lemma that the collection of measures 
$$\delta_{\mathbf{t}}\otimes \nu_{\mathbf{f}}$$ where $\mathbf{t}$ runs over all rooted spanning trees of $W$ and $\mathbf{f}$ over all elements of $F_{W,w}$
is a  linearly independent family of measures on $\mathcal T G$.

Now fix as above a forest $\mathbf{f}\in F_{W.w}$ and let $\mathbf{b}=a_{W.w}\times \mathbf{f}$. Recall that $\psi(\mathbf{b})=W$, and call $B \subset V\setminus W$ the set of boundary points relative to the tree $\mathbf{b}$, as defined in Section~\ref{sec:boundary}.
Let $H$ be the subgraph of $\mathcal T G$ where we have erased all edges having for  source a tree rooted in a vertex of $B$. Let $K$ be the subset of vertices of $H$ which can be reached by a path in $H$ starting from a tree of the form $\mathbf{t}\times \mathbf{f}_{\frak F}$, for some spanning tree $\mathbf{t}$ of $W$ and some $\frak{F}\subset \frak{E}_\mathbf{b}$. Let $J\subset K$ be the subset of  trees whose root is not an element of $B$. Note that $B$, $H$, $K$ and $J$ all depend on the choice of $\mathbf{f}$ (or $\mathbf{b}$) even though we do not indicate it in the notation.

\begin{lemma}\label{dimension}
 Let $\mathcal E_\mathbf{f}$ be the space of measures spanned by $\delta_\mathbf{t}\otimes \nu_{\mathbf{f}}\mathcal L^n$ for all spanning trees $\mathbf{t}$ of $W$ and all $ n\geq 0$, then every measure in this $\mathcal E_\mathbf{f}$ is supported on the set $J$.
\end{lemma}
\begin{proof}
It is enough to prove that for all $\mathbf{t}$ and $n\geq 0$,  the measure $\delta_\mathbf{t}\otimes \nu_{\mathbf{f}}\mathcal L^n$ has support in $J$, since this property is preserved under taking linear combinations.
 Let us compute
$\delta_\mathbf{t}\otimes \nu_{\mathbf{f}}\mathcal L^n(\mathbf{c})$ for 
 a tree $\mathbf{c}$ rooted a some boundary point $v\in B$. Recall that boundary points and erased points coincide by Lemma \ref{erase}.
One has 
\begin{equation}\label{J}
\delta_\mathbf{t}\otimes \nu_{\mathbf{f}}\mathcal L^n(\mathbf{c})=\sum_{\frak F\subset \frak E}(-1)^{|\frak F|}\sum_\pi \mathcal L_\pi
\end{equation}
where the sum $\sum_\pi \mathcal L_\pi$ is over all paths $\pi$ of length $n$ in 
$\mathcal{T}G$
 starting at 
 $\mathbf{t}\times \mathbf{f}_{\frak F}$ and ending in $\mathbf{c}$
 and $\mathcal L_\pi$ is the product of $\mathcal L_e$ over all edges $e$ traversed by $\pi$. Let $\pi$ be such a path and $\tau$ its projection on $G$, then the quantity $\mathcal L_\pi$ is equal to $L_\tau$. Assume that 
$v\notin \frak F$ then the path $\tau$ can be lifted to a path $\pi'$ starting at 
$\mathbf{t}\times f_{\frak F\cup\{v\}}$.  The only difference between $f_{\frak F}$ and $f_{\frak F\cup\{v\}}$ is the edge coming out of $v$. Since the path $\tau$ ends in $v$, the edge starting from  $v$ is deleted in the end tree of $\pi'$, therefore the end point of $\pi'$ is  again $\mathbf{c}$.
 It follows that the contributions of 
$\mathcal L_\pi$ and $\mathcal L_{\pi'}$ to the sum cancel. If 
$v\in \frak F$ we consider the path $\pi'$ started at 
$\mathbf{t}\times \mathbf{f}_{\frak F\setminus\{v\}}$, again the two contributions cancel. It follows that any contribution to the right hand side of (\ref{J}) comes with another which cancels it, therefore the quantity
$\delta_\mathbf{t}\otimes \nu_{\mathbf{f}}\mathcal L^n(\mathbf{c})$ vanishes for all $n$ and all 
trees  $\mathbf{c}$ rooted in some boundary point.

 Let now $\mathbf{c}$ be a tree which does not belong to the set $K$.
 We prove that
\begin{equation}\label{K}
\delta_\mathbf{t}\otimes \nu_{\mathbf{f}}\mathcal L^n(\mathbf{c})=0
\end{equation}
by induction on $n$. Clearly this is true if $n=0$ and
$$
\delta_\mathbf{t}\otimes \nu_{\mathbf{f}}\mathcal L^{n+1}(\mathbf{c})=\sum_{\mathbf{d}}
\delta_\mathbf{t}\otimes \nu_{\mathbf{f}}\mathcal L^{n}(\mathbf{d})\mathcal L_{\mathbf{d}\mathbf{c}}.$$
Since  $\mathbf{c}\notin K$, if $(\mathcal L)_{\mathbf{d}\mathbf{c}}\ne 0$ then either 

i) $\mathbf{d}$ is rooted in a boundary point, or 

ii) $\mathbf{d}\notin K$. 

In the first  case $\delta_\mathbf{t}\otimes \nu_{\mathbf{f}}\mathcal L^{n}(\mathbf{d})=0$ by the first part of the proof.

In case ii) $\delta_\mathbf{t}\otimes \nu_{\mathbf{f}}\mathcal L^{n}(\mathbf{d})=0$ follows from the induction hypothesis.

Equation (\ref{K}) follows.

\end{proof}

Now we let $\mathcal{E}$ be the span of the spaces $\mathcal{E}_\mathbf{f}$ for all forests $\mathbf{f}\in F_{W,w}$. Equivalently $\mathcal E$ is the space  of measures spanned by $[\delta_\mathbf{t}\otimes \nu_{\mathbf{f}}]\mathcal L^n$ for all $\mathbf{t}$, $ n$ and $\mathbf{f}$.
By construction the space $\mathcal E$ is invariant by the action of $\mathcal L$ on measures.

\begin{lemma}
The subspace $\mathcal F$ of $\mathcal E$ which consists of measures supported by trees with root not in $W$ is an invariant subspace.
\end{lemma}
\begin{proof}
It is enough to prove that for each $\mathbf{f}\in F_{W,w}$ the subspace of $\mathcal{E}_\mathbf{f}$ which consists of measures supported by trees with root not in $W$ is an invariant subspace. 
This is clear from the last lemma, since in the graph $H$ we have suppressed edges coming out from vertices of $B$ (boundary vertices), hence it is not possible for a path to come back in $W$ after having left it.\end{proof}

\begin{lemma}\label{lemma:quotient}
The action of $\mathcal L$ on the quotient space $\mathcal E/\mathcal F$ carries $m(W,w)$ copies of $\mathcal L_{(W)}$. 
\end{lemma}
\begin{proof}
Indeed for each forest $\mathbf{f}$ 
in $F_{W,w}$
and any $\mathbf{b}$ spanning rooted tree of $W$, the measure  $\delta_\mathbf{b}\otimes\nu_\mathbf{f}$  satisfies 
$[\delta_\mathbf{b}\otimes\nu_\mathbf{f}]\mathcal L=[\delta_\mathbf{b}\mathcal L_{(W)}]\otimes \nu_\mathbf{f}+\chi$ where $\chi\in\mathcal F$. 
Moreover the space $span(\nu_{\mathbf{f}})$ has dimension $m(W,w)$ by lemma \ref{dimension}. The lemma follows.
\end{proof}

We can now finish the proof of the main results.
\begin{proof}[Proof of Definition-Lemma~\ref{mW} and~Theorem~\ref{thm:main}] From Lemma~\ref{lemma:quotient}, it follows that $\det(\mathcal L)$ is divisible by $$\det(L_W)^{m(W,w)}$$ for any strongly connected $W$ and $w\in W$. In particular we can take $m(W,w)$ to be maximal among all $w$ in $W$. This implies, since the different $\det(L_W)$ are prime polynomials (see~Lemmas~\ref{W_irred}) that 
\begin{equation}\label{R}
\det(\mathcal L)
\end{equation}
 is divisible by 
\begin{equation}\label{prod}\det(L)\times\prod_{W\mbox{ s.c.}}\det(L_W)^{\max_{w\in W}(m(W,w))}
\end{equation}
Now, the degree of (\ref{R}) is $|\mathcal T V|$ while that of 
(\ref{prod}) is  $\sum_W |W|\max_{w\in W}(m(W,w))$, therefore
$$|\mathcal T V|\geq \sum_W |W|\max_{w\in W}(m(W,w)).$$
By definition of $m(W,w)$ we have:
$$|\mathcal T V|=\sum_W\sum_{w\in W}m(W,w).$$
It follows that  we have equality
for all $w\in W$:
$$m(W,w)=\max_{w\in W}(m(W,w)).$$ 
This proves that $m(W,w)$ does not depend on $w\in W$ and justifies the notation $m(W)$. This also proves that $m(W)$ is the multiplicity of the prime factor $\det(L_W)$ in $\det(\mathcal L)$. This quantity does not depend on the order chosen on $V$, thus justifying Definition-Lemma~\ref{mW}.

We have thus proved that the two sides of \eqref{main1} are scalar multiples of each other. The proportionality constant is easily seen to be $1$ by looking at the top degree coefficient in the variables $y$.
\end{proof}

\section{The case of multiple edges}\label{multiple}
Although Theorem~\ref{thm:main} only covers the case of \emph{simple} directed graphs, it is easy to use it to address the case of multiple edges.  
Indeed there is a well-known trick which produces a directed graph with no multiple edges, starting from an arbitrary directed graph, which consists in adding a vertex in the middle of each edge of the original graph. These new vertices have one incoming and one outgoing edge, obtained by splitting the original edge.
 This produces a new graph $\tilde G=(\tilde V,\tilde E)$ with $|\tilde V|=|V|+|E|$ and
$|\tilde E|=2|E|$. Given a vertex $v\in V$ there is a natural bijection between  spanning trees of $G$ and of $\tilde G$ rooted at $v$. For a vertex $\tilde v$ of the new graph sitting on an edge $e$ with $s(e)=v$ of $G$, there is a natural bijection with the spanning trees rooted at $v$. Thus the graph $\mathcal T\tilde G$ is obtained from $\mathcal TG$ by adding vertices in the middle of the edges.
It is now an easy task to transfer results on $\tilde G$ to results on $G$.
We leave the details to the interested reader (the examples of the next section may serve as a guideline for this).

Note that we do not need to take care of loops, that are irrelevant to the study of spanning trees.

\section{Examples and applications}
\label{sec:examples}

In this section we illustrate our result on a few simple examples.

\subsection{The cycle graph}
\label{sec:ring}
This example was treated in \cite{B}, let us see how to recover it via our main result. Let $G=(V,E)$ be the \emph{cycle graph} of size $n$, with vertex set $V=[1..n]$ and a directed edge from $i$ to $j$ if $j=i\pm 1 \mod n$. Thus $G$ has $n$ vertices and $2n$ directed edges. 
The graph $G$ has $n^2$ spanning trees: a spanning tree $\mathbf{a}$ is characterized by its root vertex $r\in [1..n]$ and by the unique $i\in [1..n]$ such that $\{i,i+1\} \mod n$ are the  two vertices of degree $1$ in the tree. 

We note that for any subset of vertices $W\subset V$ of cardinality $n-1$, one has $m(W)=1$. To see this, recall that $m(W)=m(W,w)$ for any $w \in W$ and choose for $w$ a neighbour of the unique vertex $u$ not in $W$: then it is clear that the only spanning tree $\mathbf{a}$ such that $\psi(\mathbf{a})=W$ is the one rooted at $w$ in which $u$ and $w$ have degree $1$.
It is then easy to see, either directly or by considering the degree of \eqref{main1}, that these are the only proper subsets $W\subsetneq V$ such that $m(W)\neq 0$.

Applying Theorem~\ref{thm:mainSpanning}, we obtain that $\Phi_G$ is the product of all symmetric minors of $Q$ of size $n-1$, which was Theorem~2 in \cite{B}.

\subsection{The complete graph (spanning trees of the graph of all Cayley trees)}
\label{sec:complete}
If $G=K_n$ is the complete graph on $n\geq 1$ vertices, then $\mathcal{T}G$ is the set of all rooted Cayley trees of size $n$, thus $\mathcal{T}G$ has $n^{n-1}$ vertices by Cayley's formula~\eqref{eq:cayley}.
If $\mathbf{a}$ is a Cayley tree rooted at $r\in [1..n]$, applying the exploration algorithm to $\mathbf{a}$ has the following effect: at the first step, all neighbours of $r$ in $\mathbf{a}$ are explored and added to $A$, and all other vertices of $V\setminus \{r\}$ are erased. It follows that for any $W\subset [1..n]$ and $w\in W$,  the multiplicity $m(W,w)$ is equal to the number of Cayley trees rooted at $w$ in which the root has 1-neighbourhood $W\setminus \{w\}$. Those trees are in bijection with spanning forests of $[1..n]\setminus\{w\}$ rooted at $W\setminus\{w\}$. We obtain, using a classical formula for the number of labeled forests of size $n-1$ rooted at $k-1$ fixed roots:
$$
m(W)=m(w,W)=
(k-1) (n-1)^{n-k-1}, \mbox{ where }k=|W|. 
$$

This formula for the multiplicity appeared as a conjecture by the second author in~\cite{B}. It is however easily seen to be equivalent to an earlier result of Athanasiadis~\cite[Corollary 3.2]{An}, which also refers to an earlier conjecture of Propp (we were not aware of the reference~\cite{An} at the time~\cite{B} was written).
By applying Theorem~\ref{thm:mainSpanning} we obtain that the number of spanning trees of the graph $\mathcal{T}K_n$ is equal to:
$$
n^{n-2} \prod_{k=1}^{n-1} \left((n-k)n^{k-1}\right)^{(k-1) (n-1)^{n-k-1}{n \choose k} } \;.
$$
It would be interesting to give a direct combinatorial proof of this formula.

\subsection{Bouquets, and the hypercube.}
\label{sec:bouquets}
 Fix $k\geq 1$ and integers $n_1,n_2,\dots, n_k \geq 1$. Consider the \emph{bouquet graph} $B$ with vertex set 
$$V=\{0,1,\dots,k\}\uplus\{v_{i}^j, 1\leq i \leq k, 1\leq j\leq n_i\}$$
and a directed edge between each $v_{i}^j$ and $0$, between each vertex $i$ and each $v_{i}^j$ for $1\leq j \leq n_i$ and between $0$ and each vertex $i$ in $[1..k]$. See the following picture:
\begin{center}
\includegraphics[width=0.4\linewidth]{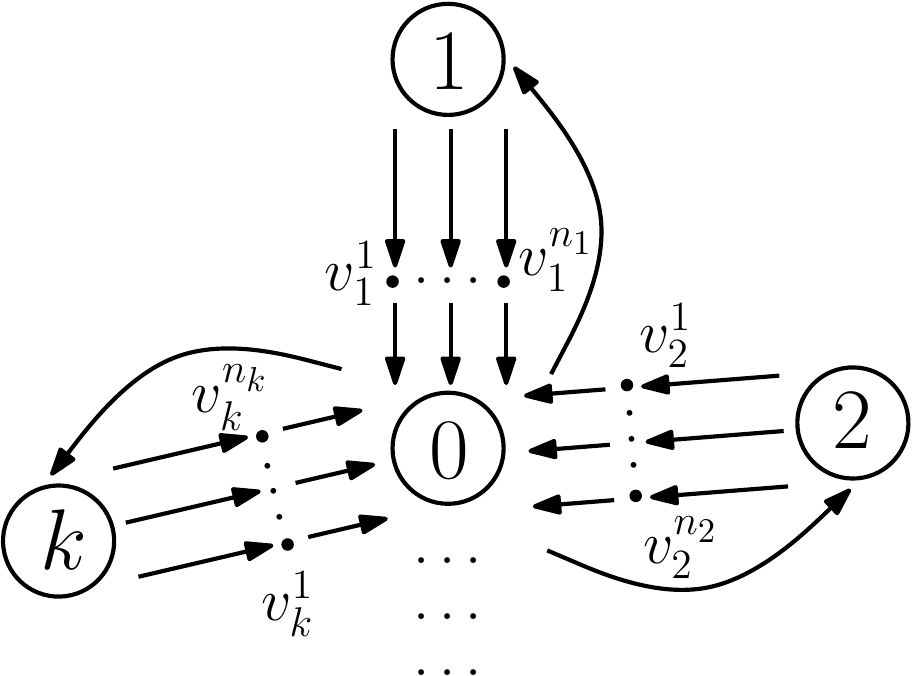}
\end{center}
For $1\leq i \leq k$, $1\leq j \leq n_i$,
we assign the weight $x_i^j$ to the edge entering the vertex~$v_i^j$, the weight $s_i$ to the edge going from $0$ to the vertex $i$,  and we assign the weight $1$ to all other edges.
A spanning tree of $B$ rooted at $0$ is naturally parametrised by the index in $[1..n_i]$ of the edge outgoing from each vertex $i$ in $[1..k]$. We let $\mathbf{a}_{\underline{m}}$ be the spanning tree rooted at $0$ naturally parametrized by $\underline{m} \in [1..n_1]\times[1..n_2]\times \dots \times [1..n_k]$.
For each $\underline{m}$, the tree $\mathbf{a}_{\underline{m}}$ has $k$ outgoing edges in $\mathcal{T}B$, to trees that we note $b_{\underline{m}'}^i$ for $i\in [1..k]$, where $b_{\underline{m}}^i$ is rooted at the vertex $i$ and $\underline{m}'$ is the projection of $\underline{m}$ to $[1..n_1]\times[1..n_2]\times \dots \times \widehat{[1..n_i]} \times \dots \times [1..n_k]$ ($i$-th set in the product omitted). Each tree $b_{\underline{m}'}^i$ has an outgoing path of length $2$ going to each tree $\mathbf{a}_{\underline{m}}$ such that $\underline{m}$ projects to $\underline{m}'$.  
For example if $k=1$, then $\mathcal{T}B$ is the following ``star graph'':
\begin{center}
\includegraphics[width=0.35\linewidth]{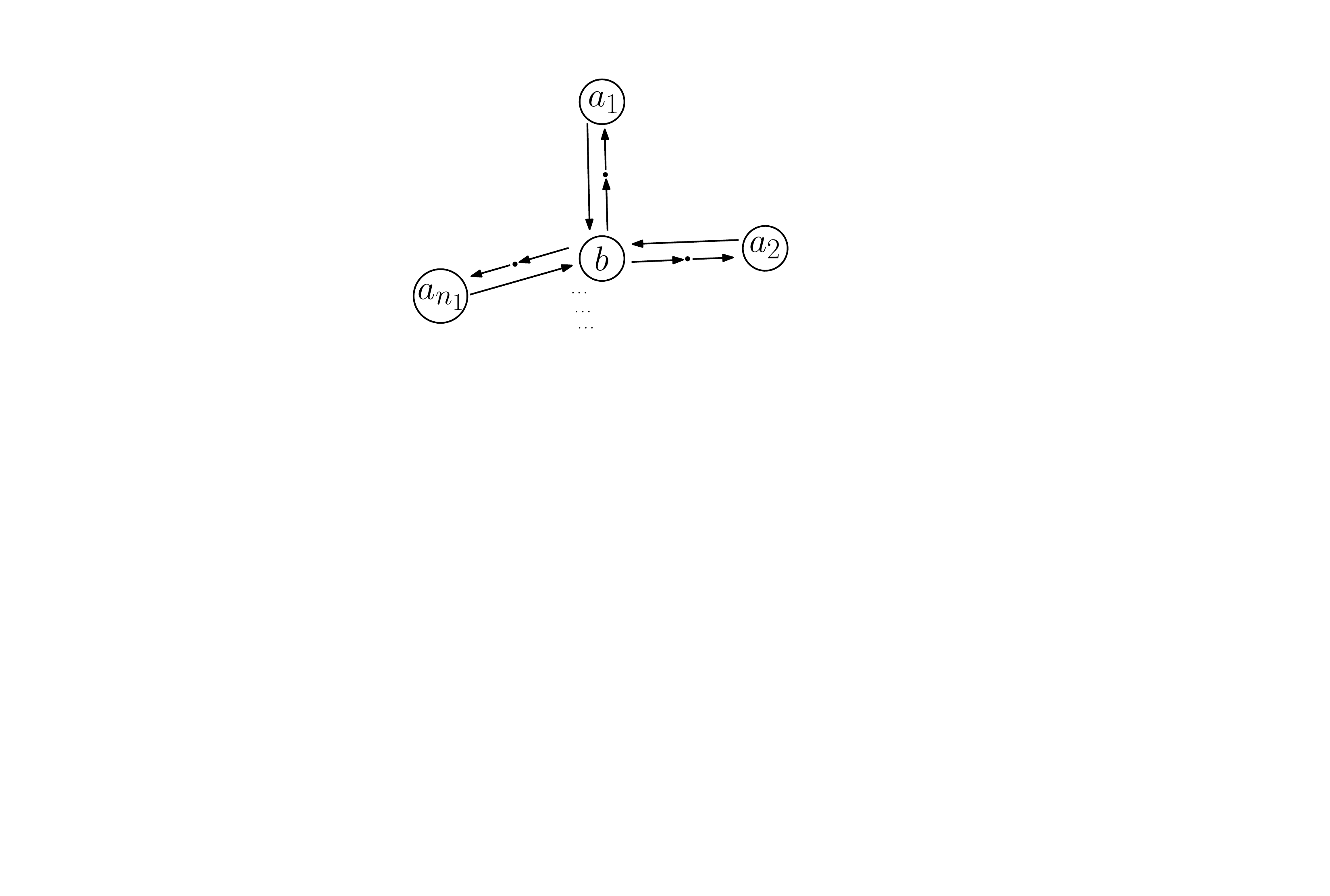}
\end{center}
For $k\geq 1$, $\mathcal{T}B$ can be interpreted as a ``partial product'' of such star graphs of parameters $n_1,n_2,\dots, n_k$, more precisely it is the subgraph of the product of these graphs induced by the subset of vertices that are such that at most one of their coordinates is a vertex which is not ``of type a''. In particular, if $n_1=n_2=\dots=n_k=2$, the graph $\mathcal{T}B$ is isomorphic to the hypercube $\{0,1\}^k$, in which three vertices are inserted in each edge, and the edge is duplicated into six directed edges as in Figure~\ref{fig:hypercube}. The mapping between $\mathcal{T}B$ and the hypercube sends the tree $\mathbf{a}_{\underline{m}}$ to the point $\underline{m}$, while the tree $\mathbf{b}_{\underline{m}'}^i$ is interpreted as the vertex lying in the middle of the edge of the hypercube defined by the vector $\underline{m}'$ (this edge points in the $i$-th axial direction).

\begin{figure}[h!!!!]
\includegraphics[width=0.4\linewidth]{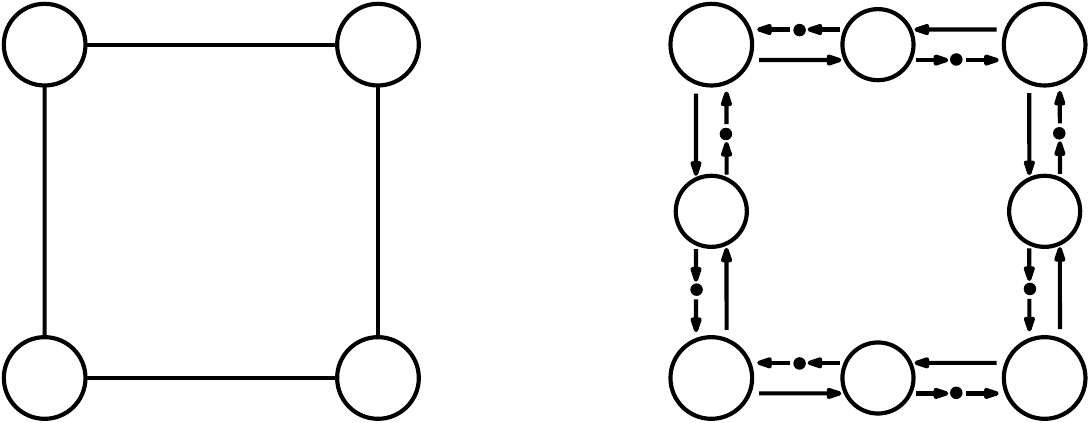}
\caption{Left: The hypercube $\{0,1\}^2$; Right: the graph $\mathcal{T}B$ for $k=2$ and $n_1=n_2=2$.}\label{fig:hypercube}
\end{figure}

Let us now apply Theorem~\ref{thm:main} to this example.  For each $I\subset [1..k]$, let $W_I$ be the strongly connected subset of $B$ consisting of $0$ and all vertices in the $i$-th petal of the bouquet for some $i \in I$. It is easy to see that these sets are the only ones with nonzero multiplicity. By basic counting, it is immediate to see that $m(W_I)=m(W_I,0) = \prod_{i\not\in I} (n_i-1)$, since a spanning tree $\mathbf{a}$ of $B$ rooted at $0$ is such that $\psi(\mathbf{a})=W_I$ if and only if the edge outgoing from the vertex $i$ is (resp. is not) the one with smallest outgoing vertex for each $i \in I$ (resp. $i\not\in I$). 
Moreover it follows from the interpretation in terms of rooted forests (Kirchoff's theorem) that for $i\in [1..k]$ one has
\begin{align}\label{eq:detQWI}
\det Q_{W_I}= \left(\sum_{i\in[1..k]\setminus I}s_i\right) \times \prod_{i\in I} \sum_{j=1}^{n_i}x_i^j.
\end{align}
From Theorem~\ref{thm:mainSpanning} one thus obtains the value of the polynomial~$\Phi_B$:
$$
\Phi_B=
\prod_{I\subsetneq [1..k]} \left(\sum_{i\in[1..k]\setminus I}s_i\right) \prod_{i\in I}\left(\sum_{j=1}^{n_i} x_i^j\right)^{\prod_{i\not\in I} (n_i-1)}.
$$ 
\noindent Equation\eqref{eq:polbiane} then implies that for any $\underline{m} \in [1..n_1]\times[1..n_2]\times \dots\times[1..n_k]$ the generating polynomial of spanning trees of $\mathcal{T}B$ rooted at $\mathbf{a}_{\underline{m}}$ is given by:
\begin{align}\label{eq:Zhyp}
Z=\left(\prod_i x_i^{m_i} \right)\Phi_B.
\end{align}
Let us now examine more precisely the case $n_1=n_2=\dots=n_k=2$ and the link with the hypercube. Let $Z_{\underline{m}}\equiv Z_{\underline{m}}(y_i^0,y_i^1, t_i, 1\leq i \leq k)$ be the generating polynomial of spanning trees of the hypercube $\{0,1\}^k$ rooted at $\underline{m}$, where $y_i^j$ marks the number of edges in the tree mutating the $i$-th coordinate to the value $j$, and $t_i$ marks the number of edges of $\{0,1\}^k$ that are parallel to the $i$-th axis and are \emph{not} present in the tree. Then it is easy to see combinatorially (see Figure~\ref{fig:hypercube} again) that we have:
\begin{align}\label{eq:Zhyp2}
Z = Z_{\underline{m}}(s_ix_i^1; s_ix_i^2; x_i^1+x_i^2).
\end{align}
Therefore the value of the generating polynomial $T_{\underline{m}}(y_i^0,y_i^1)=Z(y_i^0, y_i^1, \mathbf{1})$ can be recovered via the (invertible) change of variables $x_i^1+x_i^2=1$, $s_i x_i^1=y_i^0$, and $s_i x_i^2=y_i^1$, \textit{i.e.} by substituting $x_i^{1} \leftarrow \frac{y_i^{0}}{y_i^0+y_i^1}$,  $x_i^{2}\leftarrow \frac{y_i^{1}}{y_i^0+y_i^1}$, and $s_i \leftarrow y_i^0+y_1^1$ in~\eqref{eq:Zhyp2}. We finally obtain the generating polynomial of spanning trees of the hypercube rooted at $\underline{m}$:
\begin{align}
T_{\underline{m}}(y_i^0,y_i^1)&=
\prod_{i=1}^k \frac{y_i^{m_i}}{y_i^0+y_i^1}
\prod_{I\subsetneq [1..k]} \left(\sum_{i\in[1..k]\setminus I}y_i^0+y_i^1\right) \nonumber\\
&=\prod_{i=1}^k y_i^{m_i}
\prod_{J\subset [1..k]\atop |J|\geq 2} \left(\sum_{i\in J}y_i^0+y_i^1\right),
\label{eq:spanningH}
\end{align}
in agreement with \cite[Eq~(13)]{Bernardi} (see also~\cite[Thm. 3]{MR}).

We note that a more refined enumeration can be obtained. First, let us now assign the weight $w$ (instead of $1$) to all the edges leaving the vertices $v_i^j$, and let us replace the weights $x_i^j$ by $w x_i^j$. Using Kirchoff's theorem and a careful enumeration of spanning forests of $B$, one can generalize~\eqref{eq:detQWI} and prove that for $I\subsetneq [1..k]$ the determinant $\det (zI-Q_{W_I})$ is equal to:
\begin{align*}
&\left(z+\sum_{i\not\in I} s_i\right) \prod_{i\in I}\left(z+\sum_{j=1}^{n_i} w x_i^j\right)(w+z)^{n_i}\\
&+\sum_{i_0\in I} s_{i_0} \left( z \sum_{j=1}^{n_{i_0}} wx_{i_0}^j(w+z)^{n_{i_0}-1}+z(w+z)^{n_{i_0}}\right)
\prod_{i\in I\setminus\{i_0\}}\left(z+\sum_{j=1}^{n_{i}}wx_{i}^j\right)
(w+z)^{n_i}.
\end{align*}
This enables to apply Theorem~\ref{thm:main} and obtain the full generating polynomial of forests of the graph $\mathcal{T}B$. By extracting the top degree coefficient in $w$ in the obtained formula, we obtain the generating function of spanning forests of $\mathcal{T}B$ in which roots can only be vertices ``of type a''. In the case $n_1=n_2=\dots=n_k=2$, recalling that $m(W_I)= 1$ for all $I$, we obtain for this quantity the formula:
$$
\prod_I
\left(z+\sum_{i\not\in I} s_i\right)
\prod_{i\in I} \sum_{j=1}^2 x_i^j.
$$  
Now, the generating polynomial of directed forests on $\mathcal{T}B$ that have only roots of ``type a'', and of spanning forests of the hypercube $\{0,1\}^k$ are related combinatorially by the same combinatorial change of variables as above, namely $x_i^1+x_i^2=1$, $s_i x_i^1=y_i^0$, and $s_i x_i^2=y_i^1$, that implies in particular that $s_i=y_i^0+y_i^1$. We thus obtain:
\begin{cor}[{\cite[Eq~(3)]{Bernardi}}]\label{cor:hypercube}
The generating function of spanning oriented forests of the hypercube $\{0,1\}^k$, with a weight $z$ per root and a weight $y_i^j$ for each edge mutating the $i$-th coordinate to the value $j$ is given by:
$$
\prod_{J\subset [1..k]} \Big(z+ \sum_{i\in J} ( y_i^0 + y_i^1) \Big).
$$
\end{cor}
We conclude this section with a final comment. Of course, our proof of~\eqref{eq:spanningH} or Corollary~\ref{cor:hypercube} via Theorem~\ref{thm:mainSpanning} is more complicated than a direct enumeration  using Kirchoff's theorem and an elementary identification of the eigenspaces. However, it sheds a new light on these formulas by placing them in the general context of tree graphs. Moreover, this places the problem of finding a combinatorial proof of these results and of our main theorem under the same roof. An indication of the difficulty of this problem is that as far as we know, and despite the progresses of~\cite{Bernardi}, no bijective proof of~\eqref{eq:spanningH} (nor even~\eqref{eq:hypercube}) is known.

\end{document}